\documentclass[11pt]{article}
\usepackage{mathrsfs}
\usepackage{amssymb,amsfonts,amsmath,color,amsthm}

\usepackage{xcolor}
\newtheorem{thm}{Theorem}[section]
\newtheorem{lemma}[thm]{Lemma}

\newtheorem{conj}[thm]{Conjecture}

\begin{document}
\title{A local spectral condition for perfect matchings in 3-graphs\thanks{Partially supported by National Key Research and Development Program of China under grant 2023YFA1010203 and National Natural
Science Foundation of China under grant 12271425}}

\author{
Huiqiu Lin\thanks{School of Mathematical Sciences, East China University Of Science And Technology, Shanghai, 200237, China.
Email: huiqiulin@126.com}~~~~~~~
Hongliang Lu\thanks{School of Mathematics and Statistics,
Xi'an Jiaotong University, Xi'an, Shaanxi, 710049, China. 
Email: luhongliang215@sina.com.}~~~~~~~
Feihong Yuan\thanks{School of Mathematics and Statistics,
Xi'an Jiaotong University, Xi'an, Shaanxi, 710049, China.  Email: fhyuan@stu.xjtu.edu.cn.}~~~~~~~
Xiaonan Zhao\thanks{School of Mathematics and Statistics,
Xi'an Jiaotong University, Xi'an, Shaanxi, 710049, China. 
Email: x.n.zhao@foxmail.com}}

\date{}

\maketitle

\date{}

\maketitle
\begin{abstract}
 Let $\gamma$ be a constant such that $0 < \gamma < 1$, and let $n$ be a sufficiently large integer. Consider a $3$-uniform hypergraph $H$ on $n$ vertices. 
In 2013, K\"{u}hn, Osthus, and Treglown, along with Khan  independently, proved that for large enough $n$ with $n\equiv 0\pmod{3}$, if $\delta_1(H)\geq\binom{2n/3}{2}$, then $H$ admits a perfect matching. 
For any vertex $v\in V(H)$, we define $N_H(v)$ as the $2$-graph with vertex set $V(H)\setminus\{v\}$ and edge set $E(N_H(v)) = \{e\subseteq V(H)\setminus\{v\}: e\cup \{v\}\in E(H)\}$. 
In this paper, we show that if $\rho(N_H(v)) > (2/3+\gamma)n$ for all $v\in V(H)$, where $\rho(N_H(v))$ denotes the spectral radius of $N_H(v)$, then $H$ has a perfect matching. This bound is asymptotically tight. 
Furthermore, for  integer $s$ satisfying 
$n\geq 3s+3$,  we establish that if 
\[ \rho(N_H(v))>\frac{1}{2}(s-1+\sqrt{(s-1)^2+4s(n-s-1)})\] holds for every $v\in V(H),$
  then $H$ admits a fractional matching of size $s+1$. Notably, this second spectral bound is tight.
\end{abstract}


\section{Introduction}
For any positive integer $k$ and any set $S$, we write $[k]:=\{1,\ldots,k\}$ and
$
\binom{S}{k}:=\{T\subseteq S: |T|=k\}.
$
A {\it hypergraph} $H$ consists of a vertex set $V(H)$ and an edge set $E(H)\subseteq
2^{V(H)}$. A hypergraph $H$ is {\it $k$-uniform} if $E(H)\subseteq \binom{V(H)}{k}$, and a $k$-uniform hypergraph is also
called a {\it $k$-graph}. For simplicity, a 2-graph will be referred to simply as a graph. We write
$e(H):=|E(H)|$ and often identify $E(H)$ with $H$. 

A \emph{matching} in a hypergraph $H$ is a collection of pairwise disjoint
edges in $H$, and we use $\nu(H)$ to denote the maximum size of a
matching in $H$. The problem of finding a maximum matching in hypergraphs is NP-complete,
even for 3-graphs \cite{Ka72}. It is therefore natural to seek simple sufficient
conditions that guarantee large matchings or perfect matchings. One of the most extensively
studied classes of such conditions are degree conditions.
Let $H$ be a hypergraph and let $T\subseteq V(H)$.
The {\it degree}
of $T$ in $H$, denoted $d_H(T)$, is the number of edges of $H$
containing $T$. For an integer $l\ge 0$, the {\it minimum $l$-degree} of $H$ is
$
\delta_l(H):=
\min\{d_H(T): T\in \binom{V(H)}{l}\}.
$
 In particular, $\delta_0(H)=e(H)$, and $\delta_1(H)$ is often called the minimum {\it
vertex} degree of $H$.

A classical problem in extremal set
theory is to determine the maximum value of $e(H)$ when $\nu(H)$ is fixed. Erd\H{o}s
\cite{Erdos65} conjectured in 1965 that,
for positive integers $k,n,t$, if $H$ is a $k$-graph on
$n$ vertices and $\nu(H) < t$, then
\[
e(H)\leq \max \left\{\binom{kt-1}{k}, \binom{n}{k}-\binom{n-t+1}{k}\right\}.
\]
This bound is tight, as shown by the complete $k$-graph on $kt-1$ vertices and by the $k$-graph on $n$ vertices in which every
edge intersects a fixed set of $t-1$ vertices. There has been intensive
recent activity on this conjecture, see for instance \cite{AFH12,alon2012nonnegative,FLM,Fr13,Fr17,HLS}. In particular, Frankl
\cite{Fr13}
proved that if $n\geq (2t-1)k-(t-1)$ and $\nu(H)<t$, then
$
e(H)\le
\binom{n}{k}-\binom{n-t+1}{k},
$
and this was further improved by Frankl and Kupavskii \cite{FK18}.

H\`an, Person, and Schacht \cite{HPS} showed that for sufficiently large $n$ with $n\equiv 0\pmod{3}$, a 3-graph $H$ on $n$ vertices has a perfect matching if its minimum vertex degree exceeds $(5/9 + \gamma)\binom{n}{2}$. 
Later, independently, Khan \cite{Kh13} and K\"uhn, Osthus, and Treglown \cite{KOT13} obtained a tight result. They proved that $H$ contains a perfect matching when its minimum vertex degree is greater than $\binom{n - 1}{2}-\binom{2n/3}{2}$, and this bound is optimal.

For a simple graph $G$ (that is, a $2$-graph), let $A(G)=(a_{ij})$ denote the adjacency matrix of $G$, where $a_{xy}=1$ if $xy\in E(G)$ and $0$ otherwise. The eigenvalues of $G$ are the eigenvalues of $A(G)$, and the largest eigenvalue of $A(G)$ is called the {\it spectral radius} of $G$ and denoted by $\rho(G)$.

Feng, Yu, and Zhang \cite{FYZ17} determined the maximum spectral radius of an $n$-vertex 2-graph with a given matching number and characterized the extremal graphs.
 Given a 3-graph $H$ and a vertex $v\in V(H)$, the \textit{link graph} $N_H(v)$ is a simple graph on the vertex set $V(H)\setminus\{v\}$ and edge set
 \[
 \{e\subseteq V(H)-v\ |\ e\cup \{v\}\in E(H)\}.
 \]
This naturally gives rise to the question of whether lower-bound constraints on $\rho(N_H(v))$ for all vertices $v$ can ensure that the 3-graph $H$ has a large matching or even a perfect matching.

In this paper, we tackle this question and demonstrate that a local spectral condition for each link graph suffices to guarantee the existence of (fractional) matchings in a 3-graph. Our main findings are presented in the following three results.

\begin{thm}\label{thm1.1}
    Let $n$ be a sufficiently large integer with $n\equiv 0\pmod 3$  and let $0<\gamma\ll 1$. Let $H$ be a $3$-graph with vertex set $[n]$. If for every vertex $v\in V(H)$,
    \[
    \rho (N_H(v))>(2/3+\gamma)\,n,
    \]
    then $H$ contains a perfect matching.
\end{thm}

The condition in Theorem \ref{thm1.1} is asymptotically tight. To see this, consider two 3-graphs. Let $H_1(s,n)$ and $H_2(s,n)$ be 3-graphs with vertex set $[n]$ and edge sets
\[
E(H_1)=\left\{e\in \binom{[n]}{3}: e\cap [s]\neq \emptyset\right\} \text{ and } E(H_2)=\left\{e\in \binom{[n]}{3}: |e\cap [2s - 1]|\geq 2\right\}
\]
respectively. Clearly, neither $H_1(n/3-1,n)$ nor ${H_2(n/3,n)}$ has a perfect matching. Note that  
\[
\min \{\rho(N_{H_1(n/3 - 1,n)}(v))\ :\ v\in [n]\}=2n/3-2,
\]
and 
\[
\min \{\rho(N_{{H_2(n/3,n)}}(v))\ :\ v\in [n]\}=2n/3-2.
\]
Furthermore, 
\[
\min \{\rho(N_{H_1(s,n)}(v))\ :\ v\in [n]\}=\frac{1}{2}\Big(s-1+\sqrt{(s-1)^2+4s(n-s-1)}).
\]

Based the the constructions of $H_1(s,n)$ and $H_2(s,n)$, we would like to propose the following conjecture. 
\begin{conj}\label{conj:local-spectral}
Let $n,s$ be integers with $n\ge 3s+3$, and let $H$ be a $3$-graph with vertex set $[n]$. Suppose that for every $v\in V(H)$,
\[
 \rho\big(N_H(v)\big)>\frac{1}{2}\Big(s-1+\sqrt{(s-1)^2+4s(n-s-1)}\Big).
\]
Then $H$ has a matching of size $s+1$. Moreover, if $n=3s+3$ and for every $v\in V(H)$ we have
\[
 \rho\big(N_H(v)\big)>2n/3-2,
\]
then $H$ contains a perfect matching.
\end{conj}

The following result shows that Conjecture \ref{conj:local-spectral} holds for small matchings. 
\begin{thm}\label{thm1.2}
    Let $n,s$ be positive integers such that $n\geq 100s$. Let $H$ be a $3$-graph with vertex set $[n]$. If for every $v\in V(H)$,
\[
 \rho(N_H(v))>\frac{1}{2}\Big(s-1+\sqrt{(s-1)^2+4s(n-s-1)}\Big),
 \]
  then $H$ has a matching of size $s+1$.
\end{thm}
We also prove a fractional matching version of Conjecture \ref{conj:local-spectral}.
\begin{thm}\label{thm1.3}
    Let $n,s$ be positive integers such that $n\geq 3s+3$. Let $H$ be a $3$-graph with vertex set $[n]$. If for every $v\in V(H)$,
\[
 \rho(N_H(v))>\frac{1}{2}\Big(s-1+\sqrt{(s-1)^2+4s(n-s-1)}\Big),
 \]
  then $H$ has a fractional matching of size $s+1$. Moreover, if $n=3s+3$ and for every $v\in V(H)$, $\rho(N_H(v))>2n/3-2$, then $H$ has a perfect fractional matching.
\end{thm}


\section{Preliminaries}
In this section, we gather several auxiliary results that will be utilized in the proofs of Theorems~\ref{thm1.1}, \ref{thm1.2}, and \ref{thm1.3}.

Stanley \cite{Stan87} provided an upper bound for the spectral radius of a graph based on the number of its edges. 
\begin{thm}[\cite{Stan87}]\label{edgespec}
    Let $G$ be a simple graph with $m$ edges. Then
    \[
    \rho(G)\le (-1+\sqrt{1+8m})/2.
    \]
\end{thm}
The following sharp upper bound on the spectral radius was obtained
by Hong, Shu and Fang \cite{hong02} and Nikiforov \cite{NIKIFOROV_2002}, independently.
\begin{thm}[\cite{hong02,NIKIFOROV_2002}]\label{hong02}
    Let $G$ be a simple connected graph with $n$ vertices, $m$ edges and minimum degree $\delta$.
    Then
    \[
    \rho(G)\leq \frac{\delta-1+\sqrt{(\delta+1)^2+4(2m-\delta n)}}{2}.
    \]
\end{thm}

We require the following Nordhaus–Gaddum type inequality concerning
the sum of the spectral radii of a graph and its complement, which was
conjectured by Nikiforov~\cite{Ni07} and subsequently established by Terpai~\cite{Ter11}.

\begin{thm}[\cite{Ter11}]\label{speccompl}
    Let $G$ be a simple graph with $n$ vertices. Then 
    \[
    \rho(G)+\rho(\overline{G})\le 4n/3-1.
    \]
\end{thm}

As a consequence of Theorems~\ref{edgespec} and~\ref{speccompl}, we obtain the following lower bound
on the number of common edges of two graphs whose spectral radii are large in total.

\begin{lemma}\label{commneig}
    Suppose that $\gamma$ is a constant with $0<\gamma<1/4$. Let $n\in \mathbb{Z}$ be sufficiently large and let $G_1,G_2$
    be two simple graphs with the same vertex set $[n]$. If 
    \[
    \rho(G_1)+\rho(G_2)\ge (4/3+\gamma)n,
    \]
    then $\left|E(G_1)\cap E(G_2)\right|\ge \gamma^2n^2/2$.
\end{lemma}

\begin{proof}
 Suppose to the contrary that $\left|E(G_1)\cap E(G_2)\right|< \gamma^2n^2/2$. Let $\mathcal{E}=E(G_1)\cap E(G_2)$. Now, assume that $\left|\mathcal{E}\right|\leq \gamma^2n^2/2$. Let $G_1' = G_1\setminus \mathcal{E}$. According to Theorem~\ref{speccompl}, we have the following inequality:
\begin{align}\label{eq1.1}
    \rho(G_1')+\rho(G_2)\leq 4n/3 - 1.
\end{align}
Let $\mathbf{x}$ and $\mathbf{y}$ be the unit eigenvectors of the adjacency matrices $A(G_1)$ and $A(G_2)$, which correspond to the spectral radii $\rho(G_1)$ and $\rho(G_2)$, respectively. Then, we get the following:
    \begin{align}\label{eq1.2}
    \mathbf{x}^TA(G_1)\mathbf{x}+\mathbf{y}^TA(G_2)\mathbf{y}
    =\rho(G_1)+\rho(G_2)\ge (4/3+\gamma)n.
    \end{align}
    On the other hand,  based on inequality (\ref{eq1.1}) and Theorem~\ref{edgespec}, we obtain the following:
    \begin{align*}
    \mathbf{x}^TA(G_1)\mathbf{x}+\mathbf{y}^TA(G_2)\mathbf{y}
&=\mathbf{x}^TA(G_1')\mathbf{x}+\mathbf{y}^TA(G_2)\mathbf{y}
         +\mathbf{x}^TA(G_1[\mathcal{E}])\mathbf{x}\\
        &\le \rho(G_1')+\rho(G_2)+(-1+\sqrt{1+4\gamma^2n^2})/2\\
        &\le 4n/3-1+\gamma n,
    \end{align*}
  which contradicts (\ref{eq1.2}). Thus, the proof is completed.
\end{proof}

The following result establishes a connection between the spectral radius and the matching number of a graph, a key relationship that will be
employed in the proof of Theorem~\ref{thm1.3}.

\begin{thm}[\cite{FYZ17}]\label{s_match}
	For any $n$-vertex graph $G$ with $\nu(G)\leq m$, we have
	\begin{itemize}
		\item  [$(i)$] if $n=3m+2$, then $\rho(G)\leq 2m$ with equality if and only if $G=K_m\vee\overline{K_{n-m}}$ or $G=K_{2m+1}\cup\overline{K_{n-2m-1}}$;
		\item  [$(ii)$] if $n>3m+2$, then $\rho(G)\leq\frac{1}{2}(m-1+\sqrt{(m-1)^2+4m(n-m)})$ with equality if and only if $G=K_m\vee\overline{K_{n-m}}$.
	\end{itemize}
\end{thm}

For any $S\subseteq V(G)$, let $G[S]$ be the subgraph of $G$ induced by $S$. The following consequence transforms
a lower bound on the spectral radius into a lower bound on the number of
edges that remain after a set of vertices is deleted.

\begin{lemma}\label{edge-spec-bound}
Let $n,s$ be positive integers such that $n\geq s+1$ and let $G$ be a graph with $n$
vertices. If
\[
\rho(G)>\frac{1}{2}\Big(s-1+\sqrt{(s-1)^2+4s(n-s)}\Big),
\]
then for any $R\subseteq V(G)$ with $|R|=r$ we have
\[
e(G-R)>\frac{1}{2}(s-r)(n-s).
\]
\end{lemma}

\begin{proof}
Fix a subset $R \subseteq V(G)$ such that $|R| = r$. If $r> s$, then $(s - r)(n - s)<0$, and the conclusion is straightforward since $e(G - R)\ge0$. If $r=0$ and $e(G)\le s(n-s)/2$, by Theorem~\ref{edgespec}, we have $$\rho(G)\le \frac{-1+\sqrt{1+4s(n-s)}}{2}
<\frac12\Big(s-1+\sqrt{(s-1)^2+4s(n-s)}\Big),$$ a contradiction. Thus, we consider that $1\leq r\leq s$ in the following.
Let $G' = G - R$, and let $H$ be the graph obtained from $G'+R$ by adding all possible edges in $R$. In other words, $H$ is the join of a complete graph $K_r$ and $G'$, where the vertex set of $K_r$ is $R$, i.e., $H = K_r+G'$. 
 Therefore,
\begin{equation}\label{eq:rhoH-lower}
 \rho(H)\ge \rho(G)>
 \frac{1}{2}\Big(s-1+\sqrt{(s-1)^2+4s(n-s)}\Big).
\end{equation}

Note that the minimum degree of the graph $H$ is at least $r$. Applying Theorem~\ref{hong02}, we get
\begin{align}\label{eq4}
\rho(H)\leq
\frac{1}{2}\left(r - 1+\sqrt{(r + 1)^2+4(2e(H)-rn)}\right).
\end{align}
Since
\[
e(H)=e(G')+\binom{r}{2}+r(n - r).
\]
we have
\begin{align*}
2e(H)-rn=2e(G')+r(n - r - 1).
\end{align*}
Combining this with \eqref{eq:rhoH-lower} and \eqref{eq4}, we can deduce that
\begin{equation}\label{eq:sqrt-ineq}
\sqrt{(r + 1)^2+4\left(2e(G')+r(n - r - 1)\right)}
>
s - r+\sqrt{(s - 1)^2+4s(n - s)}.
\end{equation}
Both sides of (\ref{eq:sqrt-ineq}) are positive. Squaring and simplifying,
we obtain
\begin{align*}
8e(G')
&>(s-r)^2+(s-1)^2+4s(n-s)+2(s-r)\sqrt{(s-1)^2+4s(n-s)}\\
&\qquad-\Big((r+1)^2+4r(n-r-1)\Big)\\
&=4(s-r)(n-s)+2(s-1)(s-r)-4r(s-r)+2(s-r)\sqrt{(s-1)^2+4s(n-s)}\\
&\geq 4(s-r)(n-s)+2(s-1)(s-r)-4r(s-r)+2(s-r)(s+1)\\
&=4(s-r)(n-s)+4(s-r)^2\\
&\geq 4(s-r)(n-s),
\end{align*}
where the second inequaliy holds since $n\geq s+1$ and $r\le s$.

Hence we have
\begin{align*}
e(G')> \frac{1}{2}(s-r)(n-s). 
\end{align*}    
This completes the proof. 
\end{proof}

\begin{lemma}[Chernoff Bound]\label{chernoff}
Suppose $X_1,\ldots, X_n$ are independent random variables taking values in $\{0, 1\}$. Let
$X=\sum_{i=1}^n X_i$ and $\mu = \mathbb{E}[X]$. Then for any $0 < \delta \leq 1$,
\[
\mathbb{P}[X \geq (1+\delta) \mu] < e^{-\delta^2\mu/3}
\quad\text{and}\quad
\mathbb{P}[X \leq (1-\delta) \mu] < e^{-\delta^2 \mu/2}.
\]
In particular, when $X \sim \operatorname{Bi}(n,p)$ and $\lambda<\frac{3}{2}np$,  then
\[
\mathbb{P}(|X-np|\geq \lambda)\leq e^{-\Omega(\lambda^2/np)}.
\]
\end{lemma}

The next theorem is the Frankl–R\"odl ``nibble'' lemma, which we use to obtain almost perfect
matchings in a $k$-graph under suitable degree and codegree conditions.

\begin{thm}[\cite{FR85}]\label{nibble}
    For every integer $k\ge 2$ and any real $\sigma >0$, there exist $\tau = \tau(k,\sigma)$ and
    $d_0 = d_0(k,\sigma)$ such that for every $n\ge D\ge d_0$ the following holds:
    Every $n$-vertex $k$-graph $H$ with $(1-\tau)D<d_H(v)< (1+\tau)D$ for any $v\in V(H)$ and
    $\Delta_2(H)<\tau D$ contains a matching covering all but at most $\sigma n$ vertices.
\end{thm}

A \textit{fractional matching} in a 3-graph $H$ is a function 
$f:E(H)\to[0,1]$ such that $\sum_{e\ni v} f(e)\le 1$ for every 
$v\in V(H)$. The size of this fractional matching is given by $\sum_{e\in E(H)} f(e)$. We denote the maximum size of a fractional matching in 
$H$  as  $\nu^*(H)$, which is referred to as the fractional matching number of $H$. A fractional matching $f$ of $H$ is called \textit{perfect fractional matching} if   $\sum_{e\ni v} f(e)= 1$ for every 
$v\in V(H)$. 

The following theorem presented by Aharoni, Holzman, and Jiang helps to control the number of edges utilized in a perfect fractional matching.  

\begin{thm}[\cite{AHJ}]\label{rainfracmat}
    Let $r\ge 2$ be an integer, and let $n$ be a positive rational number. Let
    $H_1,\dots,H_{\lceil rn\rceil}$ be $r$-graphs such that $\nu^*(H_i)\ge n$ for
    $i=1,\dots,\lceil rn\rceil$. Then there exist $e_1\in H_1,\dots,e_{\lceil rn\rceil}\in
    H_{\lceil rn\rceil}$ such that $\left\{e_1,\dots,e_{\lceil rn\rceil}\right\}$ has a fractional
    matching of size $n$.
\end{thm}

The main structural ingredient in the proof of Theorem~\ref{thm1.1} is the following absorbing lemma
for 3-graphs whose link graphs all have large spectral radii.

\begin{lemma}\label{absorbing}
    Let $n$ be a sufficiently large integer with $n\equiv 0\pmod{3}$. Let $\eta$ and $\varepsilon$ be constants satisfying $0<1/n\ll\eta\ll\varepsilon\ll\gamma$. Suppose $H$ is a 3-uniform hypergraph on vertex set $[n]$ such that $\rho\bigl(N_H(v)\bigr)>(2/3+\gamma)n$ for every vertex $v\in V(H)$. Then $H$ contains a matching $M$ of
    size at most $4\varepsilon n$ such that for any subset $S\subseteq V(H)\setminus V(M)$ with
    $\left|S\right|\le \eta n$ and $|S|\equiv 0\pmod 3$, $H[S\cup V(M)]$ contains a
    perfect matching.
\end{lemma}

\begin{proof}
   According to Theorem~\ref{edgespec}, we can conclude that $\delta_1(H)\geq \binom{2n/3}{2}$. By applying Lemma~\ref{commneig}, for any two vertices $u$ and $v$ in the vertex set $V(H)$ of the hypergraph $H$, the following inequality holds:
\begin{align}\label{commne}
    \left|N_H(u)\cap N_H(v)\right|\geq 2\gamma^2n^2.
\end{align}
A 6-element subset $A$ of the vertex set $V(H)$ (i.e., $A\in \binom{V(H)}{6}$) is defined as an absorbing set for a 3-element subset $T$ of $V(H)$ (i.e., $T\in\binom{V(H)}{3}$) if the subhypergraph $H[A]$ has a matching of size 2 and the subhypergraph $H[A\cup T]$ has a matching of size 3. For each 3-element subset $T$ of $V(H)$, we use $\mathcal{L}(T)$ to denote the collection of all absorbing sets for $T$. Then, the following results can be obtained: 

\medskip
\textbf{Claim 1.~}$\left|\mathcal{L}(T)\right|\ge \gamma^4\binom{n}{6}$ for every $T\in\binom{V(H)}{3}$.
\medskip

    Let $T=\left\{ u_1,u_2,u_3\right\}$. Since $\delta_1(H)\ge \binom{2n/3}{2}$, there are at least
    \[
    \delta_1(H)-2n>n^2/5
    \]
    2-sets in $N_H(u_1)$  that lie entirely within $V(H)\setminus T$.

    Fix one such 2-set  $U_1:=\left\{v_1,v_2\right\}$, where $v_1,v_2\in V(H)\setminus T$. Next, select a 2-set $U_2$ from $V(H)\setminus (T\cup U_1)$ such that both $U_2\cup \{u_2\}$ and
    $U_2\cup \{v_1\}$ are both edges of $H$. By (\ref{commne}),
    $|N_{H}(u_2)\cap N_{H}(v_1)|\geq  2\gamma^2n^2$.  Note that the number of edges in $H$
    containing $u_2$ and intersecting $(T\cup U_1)\setminus \{u_2\}$ is at most $4n$. Thus there
    are at least
    \[
    2\gamma^2n^2- 4n>  \gamma^2n^2
    \]
    valid choices for $U_2$.

    Finally, we choose a 2-set $U_3$ from $V(H)\setminus (T\cup U_1\cup U_2)$ such that both $U_3\cup
    \{u_3\}$ and $U_3\cup \{v_2\}$ are both edges of $H$. Since
    $|N_{H}(u_3)\cap N_{H}(v_2)|\geq  2\gamma^2n^2$, there are at least
    \[
    2\gamma^2n^2-6 n>  \gamma^2n^2
    \]
  valid choices for $U_3$.

    Note that both $H[U_1\cup U_2\cup U_3]$ and $H[T\cup U_1\cup U_2\cup U_3]$ contain a perfect
    matching. Thus, there are more than $( \gamma^2n^2)^2n^2/5> \gamma^4\binom{n}{6}$ absorbing
    sets for $T$. This complete the proof of Claim 1.

    \medskip

    Now we obtain a family $\mathcal{M}$ of subsets of $V(H)$ by selecting each 6-set uniformly and
    independently with probability
    \[
    p=\frac{\varepsilon n}{\binom{n}{6}}.
    \]

    Then by the Chernoff bound (Lemma~\ref{chernoff}), with probability $1-o(1)$ the following
    holds:

   \begin{itemize}
  \item [(i)]  $\left|\mathcal{M}\right|\le 2\varepsilon n$, and

  \item [(ii)] $\left|\mathcal{L}(T)\cap \mathcal{M}\right|\ge p\left|\mathcal{L}(T)\right|/2\ge
    \varepsilon\gamma^4n/2$ for every $T\in \binom{V(H)}{3}$.
    \end{itemize}

    \medskip
    Furthermore, the expected number of intersecting pairs of sets in $\mathcal{M}$ is at most
    \[
    \binom{n}{6}\binom{6}{1}\binom{n}{5}p^2<\varepsilon^{1.5}n.
    \]
    Thus, by Markov's inequality, we derive that, with probability at least $1/2$,

    \begin{itemize}
  \item [(iii)] $\mathcal{M}$ contains at most $2\varepsilon^{1.5}n$ intersecting pairs. 
\end{itemize}
 
    Therefore, with positive probability, the matching $\mathcal{M}$ satisfies (i), (ii) and (iii). Let
    $\mathcal{M}_0$ be derived from  $M$
 in the following way: we remove one set from each pair of intersecting sets in $M$  and delete all non-absorbing sets.
   Then $H[\bigcup_{A\in\mathcal{M}_0}A]$ contains a perfect
    matching.

    Now we show that $\mathcal{M}_0$ is the desired matching. First, it follows from (i) that
    $\left|\mathcal{M}_0\right|\le 4\varepsilon n$. Next, by (ii), for each $T\in \binom{V(H)}{3}$ we have
    \[
    \left|\mathcal{L}(T)\cap \mathcal{M}_0\right|\ge \varepsilon\gamma^4n/4.
    \]
    Given $S\subseteq V(H)\setminus V(M)$ with $\left|S\right|\le \eta n$ and
    $|s|\equiv 0\pmod 3$, partition $S$ into at most $\eta n/3$ many 3-sets. Since
    $\eta n/3<\varepsilon\gamma^4n/4$, we can greedily choose an absorbing set for each 3-set.
    Hence, $H[S\cup V(M)]$ contains a perfect matching.
\end{proof}

\section{Proofs of  Theorems 1.1, 1.3 and 1.4}

\begin{proof}[Proof of Theorem~\ref{thm1.2}]
    Let
\[
R:=\{v\in V(H): d_H(v)>3s(n-2)\}
\]
and let $r:=|R|$. Write $R=\{v_1,\ldots,v_r\}$.

We first consider the case of $r\geq s + 1$. Let $R'\subseteq R$ be a subset of size $s+1$, and denote its elements as $R':=\{v_1,\ldots,v_{s + 1}\}$. For each $i\in [s]$, we define $R_i:=\{v_1,\ldots,v_i\}$ and $R_0=\emptyset$. We can observe the following degree-related inequalities. The degree of vertex $v_1$ in the hypergraph $H$ satisfies $d_{H}(v_1)>3s(n - 2)$. The degree of vertex $v_2$ in the subhypergraph $H - R_1$ is $d_{H - R_1}(v_2)>(3s - 1)(n - 2)$. In general, the degree of vertex $v_{s+1}$ in the subhypergraph $H - R_{s}$ is $d_{H-R_{s}}(v_{s+1})>2s(n - 2)>0$.
Now, we will use a greedy algorithm to find a matching of size $s + 1$ in the hypergraph $H$. Since $d_{H - R_{s}}(v_{s+1})>0$, there exists an edge $e_1$ in the subhypergraph $H - R_s$ such that $v_{s+1}\in e_1$. 
Suppose we have already found $i$ edges $e_1\in H - R_s$, $e_2\in H-(R_{s - 1}\cup e_1)$, $\cdots$, $e_i\in H-(R_{s+1 - i}\cup \bigcup_{j = 1}^{i - 1} e_j)$ with the property that $v_{s - j+2}\in e_j$ for each $j\in [i]$.
Since the degree of vertex $v_{s - i+1}$ in the subhypergraph $H - R_{s - i}$ satisfies $d_{H - R_{s - i}}(v_{s - i+1})>(2s + i)(n - 2)$, there must exists an edge $e_{i + 1}$ that contains the vertex $v_{s-i+1}$  and avoids
$R_{s-i}\cup (\bigcup_{j=1}^i e_j)$. Continuing this process, we obtain a matching of $\{e_1,\ldots,e_{s+1}\}$ in $H$. 
Thus, we consider  $r\le s $ in the following.

\medskip
\textbf{Claim 1.~}$\nu(H-R)\ge s-r+1$.
\medskip

Suppose to the contrary that $M$  is a maximum matching in $H-R$ with 
$|M|\le s-r$. Let $V(M)$ denote the set of vertices covered by $M$. Then $|V(M)|=3|M|\le 3(s-r)$. 
Recall that every vertex in $H-R$ has degree at most $3s(n-2)$
(by the definition of $R$). Since every edge of $H-R$ intersects $V(M)$, we have
\begin{align}\label{se}
    e(H-R)\le \sum_{x\in V(M)}d_{H-R}(x)
\le 3(s-r)\cdot 3s(n-2)
< 9(s-r)sn.
\end{align}
On the other hand, for each $v\in V(H)\setminus R$, we know $d_H(v)\le 3s(n-2)$
and
\[
\rho(N_H(v))
>\frac{1}{2}\Big(s-1+\sqrt{(s-1)^2+4s(n-s-1)}\Big).
\]
Note that $N_H(v)$ has $n-1$ vertices. Thus, by Lemma~\ref{edge-spec-bound}, we have
\begin{align*}
d_{H-R}(v)&=e(N_H(v)-R)\\
&>\frac{1}{2}(s-r)\big((n-1)-s\big)\\
&=\frac{1}{2}(s-r)(n-s-1)
\end{align*}
for every $v\in V(H)\setminus R$.
Summing this inequality over all vertices in $H-R$  and dividing by 3 (since each edge in the 3-uniform hypergraph is counted three times in the degree sum), we obtain
\begin{align*}
    e(H-R)>\frac{(n-r)(s-r)(n-s-1)}{6}>9(s-r)sn,
\end{align*}
due to  $0\le r\le s$
and $n\ge 100s$, which contradicts \eqref{se}. This completes the proof of Claim 1. 

By Claim~1, there exists a matching $M \subseteq H-R $ of size $ s-r+1 $. Since every vertex in $ R $ has degree greater than $ 3s(n-2) $, we can greedily extend $ M $ to a matching of size $ s+1 $ in $ H $, following the reasoning above. Thus, $ H $ contains a matching of size $ s+1 $, as required.
\end{proof}

We will make use of the concept of a fractional vertex cover. A \emph{fractional vertex cover} of $H$ is a function $\omega: V(H) \to [0,1]$ satisfying
\[
\sum_{v \in e} \omega(v) \geq 1 \quad \text{for every } e \in E(H).
\]
The minimum value of $\sum_{v \in V(H)} \omega(v)$ over all fractional vertex covers $\omega$ is denoted by $\tau^*(H)$. By linear programming duality, we have $\nu^*(H) = \tau^*(H)$.

\begin{proof}[Proof of Theorem~\ref{thm1.3}]
Let $V(H)=[n]$ and let
$\omega:V(H)\rightarrow [0,1]$ be a minimum fractional vertex cover of $H$.
Without loss of generality, we may assume that $\omega(1)\ge \omega(2)\ge \cdots\ge \omega(n)$.
Let
\[
E'=\Big\{e\in \binom{V(H)}{3}:\ \sum_{x\in e}\omega(x)\ge 1\Big\},
\]
and let $H'$ be the $3$-graph with vertex set $V(H)$ and edge set $E(H')=E'$.
Clearly $\omega$ is a fractional vertex cover of $H'$. Since $H$ is a
subhypergraph of $H'$, it follows that $\omega$ is also a \emph{minimum}
fractional vertex cover of $H'$. By linear programming duality, we have
\[
\nu^*(H)=\tau^*(H)=\tau^*(H')=\nu^*(H').
\]
Hence it suffices to show that $H'$ has a matching of size $s+1$.

\medskip
\textbf{Claim 2.~}Let $\{a_1,a_2,a_3\},\{b_1,b_2,b_3\}\in \binom{[n]}{3}$ be ordered such that
$a_1\le a_2\le a_3$ and $b_1\le b_2\le b_3$ with $a_i\leq b_i$ for all $1\leq i\leq 3$. If $\{b_1,b_2,b_3\}\in E(H')$, then
$\{a_1,a_2,a_3\}\in E(H')$.
\medskip

Since $\omega(1)\ge\cdots\ge \omega(n)$ and $a_i\le b_i$ for each $i\in[3]$,
we have $\omega(a_i)\ge \omega(b_i)$ and hence
\[
\sum_{i=1}^3\omega(a_i)\ge \sum_{i=1}^3\omega(b_i)\ge 1.
\]
By the definition of $H'$, we can deduce that $\{a_1,a_2,a_3\}\in E(H')$.
This proves Claim~2.

Since $H$ is a subgraph of $H'$, we have
$\rho(N_{H'}(v))\ge \rho(N_H(v))$ for each vertex $v$. In particular,
\[
\rho(N_{H'}(n))
\ge \rho(N_H(n))
>\frac{1}{2}\Big(s-1+\sqrt{(s-1)^2+4s(n-s-1)}\Big).
\]
Note that $N_{H'}(n)$ contains $n-1$ vertices. Then by Theorem~\ref{s_match} and $n\geq 3s+3$, we have $\nu(N_{H'}(n))\ge s+1$.
Let $M:=\{e_1,\ldots,e_{s+1}\}$ be a matching of size $s+1$ of $N_{H'}(n)$, and 
let $U\subseteq V(H)\setminus V(M)$ with $|U|=s+1$. Write $U:=\{u_1,\ldots,u_{s+1}\}$. 
Then by Claim 2, we deduce that
\[
\mathcal{M}:=\{e_i\cup\{u_i\}\ |\ i\in [s+1]\}
\]
is a matching of size $s+1$ in $H'$. 
This completes the proof of Theorem~\ref{thm1.3}.  
\end{proof}

\begin{proof}[Proof of Theorem~\ref{thm1.1}]
    Let $\gamma',\gamma''$ be two constants such that $1/n\ll\gamma''\ll \gamma'\ll \gamma$. By Lemma \ref{absorbing}, $H$ contains a matching $M$ with $|M|\leq \gamma'n$ such that for any $S\subseteq V(H)\backslash V(M)$ with $|S|\leq \gamma'' n$ and $|S|\equiv 0\pmod 3$, $H[V(M)\cup S]$ has a perfect matching. 
    
   Let $ H'= H - V(M) $, and let $n'=n-|V(M)|$. For each $ v \in V(H') $, the graph $N_{H'}(v)$ is obtained from $N_H(v)$ by deleting  at most $3\gamma'n$ vertices. By Theorem \ref{edgespec}, we get
  \begin{align}\label{newbound}
        \rho(N_{H'}(v))\ge \rho(N_H(v))-\sqrt{6\gamma'n^2}
        >(2/3+\gamma/2)n.
    \end{align}  

   \medskip
    \textbf{Claim 3.~}$H'$ contains  $r\ge n'/\ln n'$ edge-disjoint perfect fractional  matchings, say $f_1,\dots,f_r$, such that  for any $D\in \binom{V(H')}{2}$, $\sum_{i=1}^r\sum_{D\subset e}f_i(e)\leq 3$. 
    \medskip

    By Theorem~\ref{thm1.3}, Theorem~\ref{rainfracmat} and  (\ref{newbound}), $H'$ has a perfect fractional matching. By Theorem~\ref{rainfracmat}, $H'$ has a perfect fractional matching $f_1$ such that $|M_1|\leq n'$, where $M_1:=\{e\in E(H')\ :\ f_1(e)>0\}$. 
    Let $s \ge 1$ be the maximum integer such that there exist $s$ perfect fractional  matchings $f_1,\ldots, f_s$ such that for $i\in [s]$, $M_i:=\{e\in E(H')\ :\ f_i(e)>0\}$ with $|M_i|\leq n'$ and   for any $\{i,j\}\in {[s]\choose 2}$, $M_i\cap M_j=\emptyset$ and for any $D\in {V(H')\choose 2}$, $\sum_{i=1}^s \sum_{D\subseteq e}f_i(e)\leq 3$. 

    By contradiction, suppose $s < n'/\ln n'$.
Let $$U_{s}:=\Big\{D\in \binom{V(H')}{2}\ : \sum_{i=1}^{s}\sum_{D\subseteq e}f_i(e)>2\Big\}
$$ and 
\begin{align*}
E_s:=\{e\in E(H')\ : \exists D\in U_s\mbox{ such that }D\subseteq e\}.
\end{align*}
For $w\in V(H')$, since $\sum_{i=1}^s\sum_{w\in e}f_i(e)=s$, there are at  most $s$ vertices $w'\in V(H')$ for which
$\sum_{i=1}^s \sum_{\{w,w'\}\subseteq e}f_i(e)>2$.
Consequently, for any $v\in V(H')$, $d_{H'[E_s]}(v)\leq 2sn'$. Now, for any $v\in V(H')$,  we consider the following. Let $x$  be a unit eigenvector of $A(N_{H'}(v))$ corresponding to the spectral radius $\rho(N_{H'}(v))$. 
Then
\begin{align*}
\mathbf{x}^TA(N_{H'-E_s-(\cup_{i=1}^s M_i)}(v))\mathbf{x}&=\mathbf{x}^TA(N_{H'}(v))\mathbf{x}-\mathbf{x}^TA(N_{H'[E_s\cup (\cup_{i=1}^s M_i)]}(v))\mathbf{x}\\
&>(2/3+\gamma/2)n'-\sqrt{6sn'}\quad \mbox{(by (\ref{newbound}) and  Theorem \ref{edgespec})}\\
&>(2/3+\gamma/3)n'.
\end{align*}
This implies that $\rho(N_{H'-E_s-(\cup_{i=1}^sM_i)}(v))\ge (2/3+\gamma/3)n'$. Then by Theorem~\ref{thm1.3} and Theorem~\ref{rainfracmat}, $H'-E_s-(\cup_{i=1}^s M_i)$ has a perfect fractional  matching $f_{s+1}$ such that $|M_{s+1}|\leq n'$, where $M_{s+1}:=\{e\in E(H'-E_s-(\cup_{i=1}^sM_i)\ :\ f_{s+1}(e)>0\}$. We continue this process iteratively. Eventually, we can find the desired 
$r$ perfect fractional  matchings $f_1,\ldots, f_{r}$, where $r\geq  n'/\ln n'$, such that  for any $D\in \binom{V(H')}{2}$, $\sum_{i=1}^r\sum_{D\subset e}f_i(e)\leq 3$. This completes the proof of Claim 3.
 
Define a function $h:E(H')\to [0,1]$ by $h(e)=\sum_{i=1}^{ n'/\ln n'} f_i(e)$. Now let $R$ be the random subgraph of $H'$ where each edge $e\in E(H')$
 is selected independently with probability $h(e)$. This is well-defined, since the supports of the perfect fractional matchings are pairwise disjoint. We then have the following properties:
\begin{itemize}
\item [(A1)] for any vertex $v\in V(H')$, $\mathbb{E}[d_R(v)]=\sum_{v\in e}h(e)= n'/\ln n'$;

\item [(A2)] for any $D\in\binom{V(H')}{2}$, $\sum_{D\subset e}h(e)\le 3$.
\end{itemize}

By the Chernoff bound, there exists a spanning subgraph  $H''$ of $H'$ such that 
 with probability at least $1-o(1)$, the following hold: 
\begin{itemize}
\item[(B1)] for any vertex $v\in V(H'')$, $n'(1/\ln n'-1/\ln ^2n')\le d_{H''}(v)\le n'(1/\ln n'+1/\ln ^2n')$, and

\item[(B2)] for any $D\in\binom{V(H'')}{2}$, $d_{H''}(D)\le \sqrt{n'}$.
\end{itemize}

Thus by Theorem~\ref{nibble}, $H''$ contains  a matching $M_1$ satisfying 
\[
\bigl|V(H'')\setminus V(M_1)\bigr|\le \gamma'' n'\le \gamma'' n.
\]
Let $S=V(H'')\setminus V(M_1)$. By the choice of $M$ in Lemma~\ref{absorbing},
$H[V(M)\cup S]$ admits a perfect matching $M_2$. Therefore $M_1\cup M_2$ is a
perfect matching of $H$, which completes the proof. 
\end{proof}

\bibliographystyle{plain}  
\bibliography{ref1}        
\end{document}